\documentclass[12pt,a4paper]{amsart}
\usepackage{amssymb}
\usepackage{amsmath}
\usepackage{xspace}
\usepackage[english]{babel}
\usepackage[latin1]{inputenc}
\usepackage[all,2cell]{xy}

\newcommand{\at}[1]%
            {\ensuremath{\protect\underline{\mathbf{#1}}}} 

\newcommand{\op}[1]{\ensuremath{\operatorname{#1}}}        

\newcommand{\h}[1][]                                       
 {\ifthenelse{\boolean{mmode}}%
  {$\mathrm{h}$}%
  {h\nobreakdash#1\hspace{0pt}}}

\newcommand{\comp}{\circ}          
\newcommand{\adcomp}%
  {\overset{\operatorname{ad}}{\comp}} 
\newcommand{\funcomp}%
  {\overset{\operatorname{fn}}{\comp}}

\newcommand{\sccat}
{\mathbin{\kern-1pt\raisebox{6pt}{.}\kern-5pt
\downarrow\kern-5pt\raisebox{6pt}{.}\kern-1pt}}

\newcommand{\parrow}[1]
   {\underset{{\displaystyle \raisebox{5pt}%
   {$\longleftarrow$}}}{\op{#1}}{\,}}
\newcommand{\iarrow}[1]
   {\underset{{\displaystyle \raisebox{5pt}%
   {$\longrightarrow$}}}{\op{#1}}{\,}}

\newcommand{\rest}%
{\mathnormal{\restriction}}        

\newcommand{\nin}{\not\in}         
\newcommand{\incl}{\subseteq}      
\newcommand{\bprod}{\times}        
\newcommand{\function}[4]{
            \begin{array}{@{\:}c@{\:}c@{\:}l}
                   #1 &\mor& #2 \\
                   #3 &\longmapsto& #4
            \end{array} }
\newcommand{\nfunction}[4]
    {\left\{
     \function{#1}{#2}{#3}{#4}
     \right. }

\DeclareMathOperator{\ard}{ar}     
\newcommand{\fmon}[1]
 {\ensuremath{#1^{\star}}}

\newcommand{\bb}[1]{\ensuremath
 {\lvert #1 \rvert}}
\DeclareMathOperator{\Sg}{Sg}      
\DeclareMathOperator{\E}{E}        
\DeclareMathOperator{\bconcat}
            {\curlywedge}

\newcommand{\vacio}{\ensuremath{\varnothing}}

\newcommand{\brel}{\ensuremath{\xymatrix{{}\ar@{{*}{-}{*}}[r] & {}}}}

\newcommand{\nseq}[3]{\xymatrix@1@C=16pt{#1 \ar@{>}[r]_-{\scriptscriptstyle{#2}} & #3 }}

\NoCompileMatrices              
\xymatrixrowsep = {8ex}         
\xymatrixcolsep = {10ex}        

\newdir{<= }{:a(+25)\dir{-}*:a(-25)\dir{-}*!/:a(85)+1.8pt/:a(-25)\dir{-}}
\newdir{ = }{*@{=}}
\newdir{+>}{@{|}*@{-}!/-6.5pt/@{>}*!/-13pt/{}}
\newdir{ +}{{}*!/-5pt/@{+}}
\newdir{ >}{{}*!/-10pt/@{>}}
\newdir{.-}{{}*:a(-90)h!/2.1pt/{.}}
\newdir{.->}{:a(-90)h!/2.5pt/{.}*!/4pt/\dir{-}*!/0pt/\dir{-}*%
!/-2pt/\dir{-}*!/-8pt/\dir{>}}
\newdir{-.->}%
{:a(-90)h!/2pt/{.}*!/8pt/\dir{-}*!/4pt/\dir{-}*!/0pt/\dir{-}*%
!/-5pt/\dir{-}*!/-11pt/\dir{>}}
\newdir{~>}{!/4.6pt/\dir{~}*!/1pt/\dir{-}*!/-5pt/\dir{>}}
\newdir{~~>}{!/7pt/\dir{~}*!/0pt/\dir{~}*!/-5pt/\dir{>}}
\newdir{=~>}{*!/1pt/@2{~}*!/6pt/{}*!/-6pt/@2{>}}
\newdir{<~~>}{!/21pt/\dir{<}*!/21pt/\dir{-}!/12pt/\dir{~}!/5pt/\dir{~}*!/1pt/\dir{-}*!/-3pt/\dir{>}}
\newdir{=>}{!/5pt/\dir{=}!/2.5pt/\dir{=}*!/-5pt/\dir2{>}*!/7pt/\dir{ }}
\newdir{==>}{!/-2pt/\dir{=}!/-6pt/\dir{=}%
 !/-10pt/\dir{=}*!/-18pt/\dir2{>}}
\newdir{<==>}{!\dir2{<}!/-2pt/\dir{=}!/-6pt/\dir{=}%
 !/-10pt/\dir{=}*!/-17pt/\dir2{>}}
\newdir{< }{{}*!/12pt/\dir{<}}
\newdir{-o}{!/1.3pt/{}*!/0pt/{}@{o}*!/-5pt/{}}

\newsavebox{\xymor}  
\newsavebox{\xymon}  
\newsavebox{\xyepi}  
\newsavebox{\xytn}   
\newsavebox{\xyrel}  
\newsavebox{\xycel}  
\newsavebox{\xymdf}  
\newsavebox{\xyumor} 
\newsavebox{\xydmor} 
\newsavebox{\xyomor} 
\newsavebox{\xyemor} 

\newcommand{\xynode}{\makebox[0ex]{}}
\savebox{\xymor}{\ensuremath{%
\xymatrix@1@C=19pt{\xynode \ar@{>}[r] & \xynode }}}
\savebox{\xymon}{\ensuremath{%
\xymatrix@1@C=19pt{\xynode \ar@{{ +}{-}{>}}[r] & \xynode }}}
\savebox{\xyepi}{\ensuremath{%
\xymatrix@1@C=19pt{\xynode \ar@{{}{-}{+>}}[r] & \xynode }}}
\savebox{\xytn}{\ensuremath{%
\xymatrix@1@C=19pt{\xynode \ar[r]|(.44){\object@{.-}} & \xynode
}}}
\savebox{\xyrel}{\ensuremath{%
\xymatrix@1@C=19pt{\xynode \ar@{{}{-}{-o}}[r] & \xynode }}}
\savebox{\xycel}{\ensuremath{%
\xymatrix@1@C=19pt{\xynode \ar@{=>}[r] & \xynode }}}
\savebox{\xymdf}{\ensuremath{%
\xymatrix@1@C=16pt{\xynode \ar@{}[r]|{\dir{~>}} & \xynode}}}
\savebox{\xyumor}{\ensuremath{%
\xymatrix@1@C=19pt{\xynode \ar@{{}{-}^{>}}[r] & \xynode }}}
\savebox{\xydmor}{\ensuremath{%
\xymatrix@1@C=19pt{\xynode \ar@{{}{-}_{>}}[r] & \xynode }}}
\savebox{\xyomor}{\ensuremath{%
\xymatrix@1@C=19pt{\xynode \ar@{{}{-}^{< }}[r] & \xynode }}}
\savebox{\xyemor}{\ensuremath{%
\xymatrix@1@C=19pt{\xynode \ar@{{ >}{-}{>}}[r] & \xynode }}}

\newcommand{\mor}{\usebox{\xymor}}    

\newcommand{\functor}[9]{
 \xymatrix{
    #4 \save[]+<0ex,5ex>*+{#1}="1"  \restore
      \ar[d]_{#6}  \ar@{}[rd]|{\longmapsto}
  & #5 \save[]+<0ex,5ex>*+{#3}="3"  \restore
      \ar[d]^{#7}
  \\
   #8 & #9 \ar "1";"3"^-{#2} } }
\newcommand{\functornd}[9]{
 \xymatrix{
    #4 \save[]+<0ex,5ex>*+{#1}="1"  \restore
      \ar[d]_{#6}  \ar@{}[rd]|{\longmapsto}
  & #5 \save[]+<0ex,5ex>*+{#3}="3"  \restore
  \\
   #8 & #9 \ar[u]_{#7} \ar "1";"3"^-{#2} } }
\newcommand{\functordn}[9]{
 \xymatrix{
    #4 \save[]+<0ex,5ex>*+{#1}="1"  \restore
       \ar@{}[rd]|{\longmapsto}
  & #5 \save[]+<0ex,5ex>*+{#3}="3"  \restore
      \ar[d]^{#7}
  \\
   #8  \ar[u]^{#6}  & #9 \ar "1";"3"^-{#2} } }

\newcommand{\larr}{->}
\newcommand{\rarr}{->}

\newcommand{\xfunctor}[9]{
 \xymatrix{
    #4 \save[]+<0ex,5ex>*+{#1}="1"  \restore
      \ifthenelse{\equal{\larr}{->}}{\ar[d]_{#6}}{}
      \ifthenelse{\equal{\larr}{<-}}{\ar[d];[]^{#6}}{}
      \ifthenelse{\equal{\larr}{-<}}{\ar@{< }[d]_{#6}}{}
      \ar@{}[rd]|{\longmapsto}
  & #5 \save[]+<0ex,5ex>*+{#3}="3"  \restore
      \ifthenelse{\equal{\rarr}{->}}{\ar[d]^{#7}}{}
      \ifthenelse{\equal{\rarr}{<-}}{\ar[d];[]_{#7}}{}
      \ifthenelse{\equal{\rarr}{-<}}{\ar@{< }[d]^{#7}}{}
  \\
   #8 & #9 \ar "1";"3"^-{#2} } }

\UseAllTwocells

\theoremstyle{plain}
\newtheorem{theorem}{\indent\bf Theorem}[section]
\newtheorem{proposition}[theorem]{\indent\bf Proposition}
\newtheorem{corollary}[theorem]{\indent\bf Corollary}

\theoremstyle{definition}
\newtheorem{definition}[theorem]{\indent\bf Definition}

\newtheorem*{remark}{\indent\bf Remark}

\theoremstyle{remark}

\numberwithin{equation}{section}

\begin{document}
\title[$n$-ary M-s closure operators and the m-s IrB theorem]{A characterization of the $n$-ary many-sorted closure operators and a many-sorted Tarski irredundant basis theorem}

\author[Climent]{J. Climent Vidal}
\address{Universitat de Val\`{e}ncia\\
         Departament de L\`{o}gica i Filosofia de la Ci\`{e}ncia\\
         Av. Blasco Ib\'{a}\~{n}ez, 30-$7^{\mathrm{a}}$, 46010 Val\`{e}ncia, Spain}
\email{Juan.B.Climent@uv.es}
\author[Cosme]{E. Cosme Ll\'{o}pez}
\address{Universitat de Val\`{e}ncia\\
         Departament d'\`{A}lgebra\\
         Dr. Moliner, 50, 46100 Burjassot, Val\`{e}ncia, Spain}
\email{Enric.Cosme@uv.es}

 \keywords{$S$-sorted set, delta of Kronecker, support of an $S$-sorted set, $n$-ary many-sorted closure operator, uniform many-sorted closure operator, irredundant basis for a many-sorted closure space.}
\date{May 5th, 2016}

\begin{abstract}
A theorem of single-sorted algebra states that, for a closure space $(A,J)$ and a natural number $n$, the closure operator $J$ on the set $A$ is $n$-ary if, and only if, there exists a single-sorted signature $\Sigma$ and a $\Sigma$-algebra $\mathbf{A}$ such that every operation of $\mathbf{A}$ is of an arity $\leq n$ and $J = \mathrm{Sg}_{\mathbf{A}}$, where $\mathrm{Sg}_{\mathbf{A}}$ is the subalgebra generating operator on $A$ determined by  $\mathbf{A}$. On the other hand, a theorem of Tarski asserts that if $J$ is an $n$-ary closure operator on a set $A$ with $n\geq 2$, and if $i<j$ with $i$, $j\in \mathrm{IrB}(A,J)$, where $\mathrm{IrB}(A,J)$ is the set of all natural numbers $n$ such that $(A,J)$ has an irredundant basis ($\equiv$ minimal generating set) of $n$ elements, such that $\{i+1,\ldots, j-1\}\cap \mathrm{IrB}(A,J) = \varnothing$, then $j-i\leq n-1$. In this article we state and prove the many-sorted counterparts of the above theorems. But, we remark, regarding the first one under an additional condition: the uniformity of the many-sorted closure operator.
\end{abstract}
\maketitle



\section{Introduction.}\hfill

A well-known theorem of single-sorted algebra states that, for a closure space $(A,J)$ and a natural number $n\in \mathbb{N} = \omega$, the closure operator $J$ on the set $A$ is $n$-ary if, and only if, there exists a single-sorted signature $\Sigma$ and a $\Sigma$-algebra $\mathbf{A}$ such that every operation of $\mathbf{A}$ is of an arity $\leq n$ and $J = \mathrm{Sg}_{\mathbf{A}}$, where $\mathrm{Sg}_{\mathbf{A}}$ is the subalgebra generating operator on $A$ determined by  $\mathbf{A}$. On the other hand, in~\cite{cs04}, it was stated that, for an algebraic many-sorted closure operator $J$ on an $S$-sorted set $A$, $J = \mathrm{Sg}_{\mathbf{A}}$ for some many-sorted signature $\Sigma$ and some $\Sigma$-algebra $\mathbf{A}$ if, and only if, $J$ is uniform. And, by using, among others, the just mentioned result, our first main result is the following characterization of the $n$-ary many-sorted closure operators: Let $S$ be a set of sorts, $A$ an $S$-sorted set, $J$ a  many-sorted closure operator on $A$, and $n\in \mathbb{N}$. Then $J$ is $n$-ary and uniform if, and only if, there exists an $S$-sorted signature $\Sigma$ and a $\Sigma$-algebra $\mathbf{A}$ such that $J = \mathrm{Sg}_{\mathbf{A}}$ and every operation of $\mathbf{A}$ is of an arity $\leq n$.

We turn next to Tarski's irredundant basis theorem for single-sorted closure spaces. But before doing that let us begin by recalling the terminology relevant to the case. Given an $n$ in $\mathbb{N}$, a set $A$, and a closure operator $J$ on $A$, the closure operator $J$ is said to be an $n$-ary closure operator on $A$ if $J = J^{\omega}_{\leq n}$, where $J^{\omega}_{\leq n}$ is the supremum of the family $(J^{m}_{\leq n})_{m\in\omega}$ of operators on $A$ defined by recursion as follows:  for $m = 0$, $J^{0}_{\leq n} = \mathrm{Id}_{\mathrm{Sub}(A)}$; for $m = k+1$, with $k\geq 0$, $J^{k+1}_{\leq n}(X) = J_{\leq n}\circ J^{k}_{\leq n}$, where $J_{\leq n}$ is the operator on $A$ defined, for every $X\subseteq A$, as follows:
$$
\textstyle
J_{\leq n}(X) = \bigcup\{J(Y)\mid Y\in\mathrm{Sub}_{\leq n}(X)\},
$$
where $\mathrm{Sub}_{\leq n}(X)$ is $\{Y\subseteq X\mid \mathrm{card}(Y)\leq n\}$.

Alfred Tarski in~\cite{at75}, on pp.~190--191, proved, as reformulated by S. Burris and H. P. Sankappanavar in~\cite{bs81}, on pp.~33--34, the following theorem. Given a set $A$ and an $n$-ary closure operator $J$ on $A$ with $n\geq 2$, for every $i$, $j\in \mathrm{IrB}(A,J)$, where $\mathrm{IrB}(A,J)$ is the set of all natural numbers $n$ such that $(A,J)$ has an irredundant basis($\equiv$ minimal generating set) of $n$ elements, if $i<j$ and $\{i+1,\ldots, j-1\}\cap \mathrm{IrB}(A,J) = \varnothing$, then $j-i\leq n-1$. Thus, as stated by Burris and Sankappanavar in~\cite{bs81}, on p.~33, the length of the finite gaps in $\mathrm{IrB}(A,J)$ is bounded by $n-2$ if $J$ is an $n$-ary closure operator. And our second main result is the proof of Tarski's irredundant basis theorem for many-sorted closure spaces.

\section{Many-sorted sets, many-sorted closure operators, and many-sorted algebras.}\hfill

In this section, for a set of sorts $S$ in a fixed Grothendieck universe $\boldsymbol{\mathcal{U}}$, we begin by recalling some basic notions of the theory of $S$-sorted sets, e.g., those of subset of an $S$-sorted set, of proper subset of an $S$-sorted set, of delta of Kronecker, of cardinal of an $S$-sorted set, and of support of an S-sorted set; and by defining, for an $S$-sorted set $A$, the concepts of many-sorted closure operator on $A$ and of many-sorted closure space. Moreover, for a many-sorted closure operator $J$ on $A$, we define the notions of irredundant or independent part of $A$ with respect to $J$, of basis or generator of $A$ with respect to $J$, of irredundant basis of $A$ with respect to $J$, and of minimal basis of $A$ with respect to $J$. In addition, we state that the notion of irredundant basis of $A$ with respect to $J$ is equivalent to the notion of minimal basis of $A$ with respect to $J$ and, afterwards, for a many-sorted closure space $(A,J)$, we define the subset $\mathrm{IrB}(A,J)$ of $\mathbb{N}$ as being formed by choosing those natural numbers which are the cardinal of an irredundant basis of $A$ with respect to $J$. On the other hand, for a natural number $n$, we define the concept of $n$-ary many-sorted closure operator on $A$ and provide a characterization of the $n$-ary many-sorted closure operators $J$ on $A$, in terms of the fixed points of $J$. Besides, for a set of sorts $S$, we define the concept of $S$-sorted signature, and, for an $S$-sorted signature $\Sigma$, the notion of $\Sigma$-algebra and, for a $\Sigma$-algebra $\mathbf{A}$, the concept of subalgebra of  $\mathbf{A}$ and the subalgebra generating many-sorted operator $\mathrm{Sg}_{\mathbf{A}}$ on $A$ determined by  $\mathbf{A}$. Subsequently, once defined the notion of finitely generated $\Sigma$-algebra, we state that, for a finitely generated $\Sigma$-algebra $\mathbf{A}$, $\mathrm{IrB}(A,\mathrm{Sg}_{\mathbf{A}})\neq \varnothing$.

\begin{definition}
An $S$-\emph{sorted set} is a function $A = (A_{s})_{s\in S}$ from $S$ to $\boldsymbol{\mathcal{U}}$.
\end{definition}

\begin{definition}
Let $S$ be a set of sorts. If $A$ and $B$ are $S$-sorted sets, then we will say that $A$ is a \emph{subset} of $B$, denoted by $A\subseteq B$, if, for every $s\in S$, $A_{s}\subseteq B_{s}$, and that $A$ is a \emph{proper} subset of $B$, denoted by $A\subset B$, if $A\subseteq B$ and, for some $s\in S$, $B_{s}-A_{s}\neq \varnothing$. We denote by $\mathrm{Sub}(A)$ the set of all $S$-sorted sets $X$ such that $X\subseteq A$.
\end{definition}

\begin{definition}
Given a sort $t\in S$ and a set $X$ we call \emph{delta of Kronecker for} $(t,X)$ the $S$-sorted set $\delta^{t,X}$ defined, for every $s\in S$, as follows:
\begin{equation*}
\delta^{t,X}_{s} =
\begin{cases}
X, &\text{if $s = t$;}\\
\vacio, & \text{otherwise.}
\end{cases}
\end{equation*}
For a final set $\{x\}$, to abbreviate, we will write $\delta^{t,x}$ instead of the more accurate $\delta^{t,\{x\}}$.
\end{definition}

We next define, for a set of sorts $S$, the concept of cardinal of an $S$-sorted set, for an $S$-sorted set $A$, the notion of support of $A$, and characterize the finite $S$-sorted sets in terms of its supports.

\begin{definition}
Let $A$ be an $S$-sorted set. Then the \emph{cardinal of} $A$, denoted by $\mathrm{card}(A)$, is the cardinal of $\coprod A$, where $\coprod A$, the coproduct of $A = (A_{s})_{s\in S}$, is $\bigcup_{s\in S}(A_{s}\times\{s\})$. Moreover,  $\mathrm{Sub}_{\mathrm{fin}}(A)$ denotes the set of all finite subsets of $A$, i.e., the set $\{X\subseteq A \mid \textstyle\mathrm{card}(X)< \aleph_{0}\}$, and, for a natural number $n$, $\mathrm{Sub}_{\leq n}(A)$ denotes the set of all subsets of $A$ with at most $n$ elements, i.e., the set $\{X\subseteq A \mid \textstyle\mathrm{card}(X)\leq n\}$. Sometimes, for simplicity of notation, we write $X\subseteq_{\mathrm{fin}}A$ instead of $X\in \mathrm{Sub}_{\mathrm{fin}}(A)$.
\end{definition}

\begin{definition}
Let $S$ be a set of sorts. Then the \emph{support of} $A$, denoted by $\mathrm{supp}_{S}(A)$, is the set $\{\,s\in S\mid A_{s}\neq \varnothing\,\}$.
\end{definition}

\begin{proposition}
An $S$-sorted set $A$ is finite if, and only if, $\mathrm{supp}_{S}(A)$ is finite and, for every $s\in \mathrm{supp}_{S}(A)$, $\mathrm{card}(A_{s})<\aleph_{0}$.
\end{proposition}

\begin{definition}
Let $S$ be a set of sorts and $A$ an $S$-sorted set. A \emph{many-sorted closure operator on} $A$ is a mapping $J$ from $\mathrm{Sub}(A)$ to $\mathrm{Sub}(A)$, which assigns to every $X\subseteq A$ its $J$-\emph{closure} $J(X)$, such that, for every $X,Y\subseteq A$, satisfies the following conditions:
\begin{enumerate}
\item $X\subseteq J(X)$, i.e., $J$ is extensive.

\item If $X\subseteq Y$, then $J(X)\subseteq J(Y)$, i.e., $J$ is isotone.

\item $J(J(X)) = J(X)$, i.e., $J$ is idempotent.
\end{enumerate}
Given two many-sorted closure operators $J$ and $K$ on $A$, $J$ is called \emph{smaller than} $K$, denoted by $J\leq K$, if, for every $X\subseteq A$, $J(A)\subseteq K(A)$. A \emph{many-sorted closure space} is an ordered pair $(A,J)$ where $A$ is an $S$-sorted set and $J$ a many-sorted closure operator on $A$. Moreover, if $X\subseteq A$, then $X$ is \emph{irredundant} (or \emph{independent}) \emph{with respect to} $J$ if, for every $s\in S$ and every $x\in X_{s}$, $x\nin J(X-\delta^{s,x})_{s}$, $X$ is a \emph{basis} (or a \emph{generator}) \emph{with respect to} $J$ if $J(X) = A$, $X$ is an \emph{irredundant basis with respect to} $J$ if $X$ irreduntant and a basis with respect to $J$, and $X$ is a \emph{minimal basis with respect to} $J$ if $J(X) = A$ and, for every $Y\subset X$, $J(Y)\neq A$.
\end{definition}



We next state that the notion of irredundant basis of $A$ with respect to $J$ is equivalent to the notion of minimal basis of $A$ with respect to $J$. Moreover, for a many-sorted closure space $(A,J)$, we define $\mathrm{IrB}(A,J)$ as the intersection of the  set of all natural numbers and the set of the cardinals of the irredundant basis of $A$ with respect to $J$.

\begin{proposition}
Let $(A,J)$ be a many-sorted closure space and $X\subseteq A$. Then $X$ is an irredundant basis with respect to $J$ if, and only if,  it is a minimal basis with respect to $J$.
\end{proposition}

%

\begin{definition}
Let $S$ be a set of sorts and $(A,J)$ a many-sorted closure space. Then we denote by $\mathrm{IrB}(A,J)$ the subset of $\mathbb{N}$ defined as follows:
$$
\mathrm{IrB}(A,J) = \mathbb{N}\cap \biggl\{ \mathrm{card}(X)\biggm|
\begin{gathered}
X \text{ is an irredundant basis }
\\[-3pt]
\text{ of } A \text{ with respect to } J
\end{gathered}
\biggr\}.
$$
\end{definition}

Later, in this section, after having defined, for a set of sorts $S$ and an $S$-sorted signature $\Sigma$, the concept of $\Sigma$-algebra, for a $\Sigma$-algebra $\mathbf{A} = (A,F)$, the uniform algebraic many-sorted closure operator $\mathrm{Sg}_{\mathbf{A}}$ on $A$, called the subalgebra generating many-sorted operator on $A$ determined by $\mathbf{A}$, and the notion of finitely generated $\Sigma$-algebra, we will state that, for a finitely generated $\Sigma$-algebra $\mathbf{A}$,  $\mathrm{IrB}(A,\mathrm{Sg}_{\mathbf{A}})\neq \varnothing$.

\begin{definition}
Let $A$ be an $S$-sorted set, $J$ a many-sorted closure operator on $A$, and $n$ a natural number.
\begin{enumerate}
\item We denote by $J_{\leq n}$ the many-sorted operator on $A$ defined, for every $X\subseteq A$, as follows:
      $$
       \textstyle
       J_{\leq n}(X) = \bigcup\{J(Y)\mid Y\in\mathrm{Sub}_{\leq n}(X)\}.
      $$

\item We define the family $(J^{m}_{\leq n})_{m\in \mathbb{N}}$ of many-sorted operator on $A$, recursively, as follows:
      \begin{equation*}
      J^{m}_{\leq n} =
      \begin{cases}
      \mathrm{Id}_{\mathrm{Sub}(A)}, & \text{if $m = 0$;}\\
      J_{\leq n}\circ J^{k}_{\leq n},
      & \text{if $m = k+1$, with $k\geq 0$.}
      \end{cases}
      \end{equation*}
\item We denote by $J^{\omega}_{\leq n}$ the many-sorted operator on
      $A$ that assigns to an $S$-sorted subset $X$ of
      $A$, $J^{\omega}_{\leq n}(X) = \bigcup_{m\in \mathbb{N}}J^{m}_{\leq n}(X)$.
\item We say that $J$ is $n$-\emph{ary} if $J = J^{\omega}_{\leq n}$.
\end{enumerate}
\end{definition}

\begin{remark}
Let $J$ be a many-sorted closure operator on $A$. Then $J$ is $0$-ary, i.e., $J = J^{\omega}_{\leq 0}$, if, and only if, for every $X\subseteq A$, we have that
$$
 J(X) = X\cup J(\varnothing^{S}),
$$
where $\varnothing^{S}$ is the $S$-sorted set whose $s$th coordinate, for every $s\in S$, is $\varnothing$.

We next prove that $J$ is $1$-ary, i.e., that $J = J^{\omega}_{\leq 1}$, if and only if, for every $X\subseteq A$, we have that
$$
 J(X) = J(\varnothing^{S})\cup \textstyle \bigcup_{s\in S, x\in
 X_{s}}J(\delta^{s,x}).
$$
Let us suppose that, for every $X\subseteq A$, $J(X) = J(\varnothing^{S})\cup \textstyle \bigcup_{s\in S, x\in X_{s}}J(\delta^{s,x})$. Then it is obvious that, for every $X\subseteq A$, $J(X)\subseteq J_{\leq 1}(X)$. Let us verify that, for every $X\subseteq A$, $J_{\leq 1}(X) = \bigcup\{J(Y)\mid Y\in\mathrm{Sub}_{\leq 1}(X)\}\subseteq J(X)$. Let $Y$ be an element of $\mathrm{Sub}_{\leq 1}(X)$. Then $Y = \varnothing^{S}$ or $Y = \delta^{t,a}$, for some $t\in S$ and some $a\in X_{t}$. If $Y = \varnothing^{S}$, then
$$
J(\varnothing^{S})\subseteq J(\varnothing^{S})\cup \textstyle \bigcup_{s\in S, x\in
X_{s}}J(\delta^{s,x}) = J(X).
$$
If $Y = \delta^{t,a}$, then $J(\delta^{t,a})\subseteq \bigcup_{s\in S, x\in
X_{s}}J(\delta^{s,x})$, hence
$$
J(\delta^{t,a})\subseteq J(\varnothing^{S})\cup \textstyle \bigcup_{s\in S, x\in
X_{s}}J(\delta^{s,x}) = J(X).
$$
Thus $J_{\leq 1}(X)\subseteq J(\varnothing^{S})\cup \textstyle \bigcup_{s\in S, x\in
X_{s}}J(\delta^{s,x}) = J(X)$. Therefore $J = J_{\leq 1}$. Hence, for every $m\geq 1$, $J = J^{m}_{\leq 1}$. Consequently $J$ is $1$-ary.

Reciprocally, let us suppose that $J$ is $1$-ary, i.e., that, for every $X\subseteq A$, $J(X) = \bigcup_{m\in \mathbb{N}}J^{m}_{\leq 1}(X)$. Then, obviously, we have that
$$
J(X) \supseteq J(\varnothing^{S})\cup \textstyle \bigcup_{s\in S, x\in
 X_{s}}J(\delta^{s,x}).
$$
Let us verify that, for every $m\in \mathbb{N}$, $J(\varnothing^{S})\cup \bigcup_{s\in S, x\in X_{s}}J(\delta^{s,x})\supseteq J^{m}_{\leq 1}$. Evidently $J(\varnothing^{S})\cup \bigcup_{s\in S, x\in X_{s}}J(\delta^{s,x})\supseteq J^{0}_{\leq 1}(X)\cup J^{1}_{\leq 1}(X)$. Let $k$ be $\geq 1$ and let us suppose that $J(\varnothing^{S})\cup \bigcup_{s\in S, x\in X_{s}}J(\delta^{s,x})\supseteq J^{k}_{\leq 1}(X)$. We will show that $J(\varnothing^{S})\cup \bigcup_{s\in S, x\in X_{s}}J(\delta^{s,x})\supseteq J^{k+1}_{\leq 1}(X)$. By definition we have that
$$
J^{k+1}_{\leq 1}(X) = J_{\leq 1}(J^{k}_{\leq 1}(X)) = \textstyle \bigcup\{J(Z)\mid Z\in \mathrm{Sub}_{\leq 1}(J^{k}_{\leq 1}(X))\}.
$$
Let $Z$ be an element of $\mathrm{Sub}_{\leq 1}(J^{k}_{\leq 1}(X))$. Then $Z\subseteq J^{k}_{\leq 1}(X)$. But we have that $J^{k}_{\leq 1}(X) = \bigcup\{J(Y)\mid Y\in \mathrm{Sub}_{\leq 1}(J^{k-1}_{\leq 1}(X))\}$. Therefore, for some $Y\in \mathrm{Sub}_{\leq 1}(J^{k-1}_{\leq 1}(X))$, $Z\subseteq J(Y)$. Thus $J(Z)\subseteq J(J(Y)) = J(Y)$. But $J(Y)\subseteq J^{k}_{\leq 1}(X)$. Consequently $J(Z)\subseteq J^{k}_{\leq 1}(X)$. Whence, by the induction hypothesis, $J(Z)\subseteq J(\varnothing^{S})\cup \bigcup_{s\in S, x\in X_{s}}J(\delta^{s,x})$. From this, since $Z$ was an arbitrary element of $\mathrm{Sub}_{\leq 1}(J^{k}_{\leq 1}(X))$, we infer that
$$
J^{k+1}_{\leq 1}(X) = \textstyle \bigcup\{J(Z)\mid Z\in \mathrm{Sub}_{\leq 1}(J^{k}_{\leq 1}(X))\}\subseteq J(\varnothing^{S})\cup \bigcup_{s\in S, x\in X_{s}}J(\delta^{s,x}).
$$
Thus, for every $X\subseteq A$, we have that
$$
 J(X) = J(\varnothing^{S})\cup \textstyle \bigcup_{s\in S, x\in
 X_{s}}J(\delta^{s,x}).
$$
\end{remark}

\begin{remark}
Let $n$ be $\geq 1$, $A$ an $S$-sorted set, $X\subseteq A$, and $J$ a many-sorted closure operator on $A$. Then, for every $k\geq 0$ and every $Y\subseteq A$, if $Y\in \mathrm{Sub}_{\leq n}(J^{k}_{\leq n}(X))$, then $Y\in \mathrm{Sub}_{\leq n}(J^{k+1}_{\leq n}(X))$.
\end{remark}

We next state, for a natural number $n\geq 1$ and a many-sorted closure operator $J$ on an $S$-sorted set $A$, that the family of many-sorted operators $(J^{m}_{\leq n})_{m\in \mathbb{N}}$ on $A$ is an ascending chain and that $J^{\omega}_{\leq n}$, which is the supremum of the above family, is the greatest $n$-ary many-sorted closure operator on $A$ which is smaller than  $J$.

\begin{proposition}
For a natural number $n\geq 1$, an $S$-sorted set $A$, and a many-sorted closure operator $J$ on $A$, the family of many-sorted operators $(J^{m}_{\leq n})_{m\in \mathbb{N}}$ on $A$ is an ascending chain, i.e., for every $m\in \mathbb{N}$, $J^{m}_{\leq n}\leq J^{m+1}_{\leq n}$. Moreover, $J^{\omega}_{\leq n}$ is the greatest $n$-ary many-sorted closure operator on $A$ such that $J^{\omega}_{\leq n}\leq J$.
\end{proposition}

We next provide a characterization of the $n$-ary many-sorted closure operators $J$ on an $S$-sorted set $A$ in terms of the fixed points $X$ of $J$ and of its relationships with the $J$-closures of the subsets of $X$ with at most $n$ elements.

\begin{proposition}
Let $A$ be an $S$-sorted set, $J$ a many-sorted closure operator on $A$, and $n$ a natural number. Then $J$ is $n$-ary if, and only if, for every $X\subseteq A$, if, for every $Z\in \mathrm{Sub}_{\leq n}(X)$, $J(Z)\subseteq X$, then $J(X) = X$ (i.e., if, and only if, for every $X\subseteq A$, $X$ is a fixed point of $J$ if $X$ contains the $J$-closure of each of its subsets with at most $n$ elements).
\end{proposition}

\begin{proof}
If $n = 0$, then the result is obvious. So let us consider the case when $n\geq 1$.
Let us suppose that $J$ is $n$-ary and let $X$ be a subset of $A$ such that, for every $Z\in \mathrm{Sub}_{\leq n}(X)$, $J(Z)\subseteq X$. We want to show that $J(X) = X$. Because $J$ is extensive, $X\subseteq J(X)$. Therefore it only remains to show that $J(X)\subseteq X$. Since, by hypothesis, $J(X) = \bigcup_{m\in \mathbb{N}}J^{m}_{\leq n}(X)$, to show that $J(X)\subseteq X$ it suffices to prove that, for every $m\in \mathbb{N}$, $J^{m}_{\leq n}(X)\subseteq X$.

For $m = 0$ we have that $J^{0}_{\leq n}(X) = X\subseteq X$.

Let us suppose that, for $k\geq 0$, $J^{k}_{\leq n}(X)\subseteq X$. Then we want to show that $J^{k+1}_{\leq n}(X)\subseteq X$. But, by definition, we have that
$$
\textstyle
J^{k+1}_{\leq n}(X) = J_{\leq n}(J^{k}_{\leq n}(X)) = \bigcup\{J(Y)\mid Y\in \mathrm{Sub}_{\leq n}(J^{k}_{\leq n}(X))\}.
$$
Hence what we have to prove is that, for every $Y\in \mathrm{Sub}_{\leq n}(J^{k}_{\leq n}(X))$, $J(Y)\subseteq X$. Let $Y$ be a subset of $J^{k}_{\leq n}(X)$ such that $\mathrm{card}(Y)\leq n$. Since $J^{k}_{\leq n}(X)\subseteq X$, we have that $Y\subseteq X$ and $\mathrm{card}(Y)\leq n$, therefore $J(Y)\subseteq X$. Consequently, for every $X\subseteq A$, if, for every $Z\in \mathrm{Sub}_{\leq n}(X)$, $J(Z)\subseteq X$, then $J(X) = X$.

Reciprocally, let us suppose that, for every $X\subseteq A$, if, for every $Z\in \mathrm{Sub}_{\leq n}(X)$, $J(Z)\subseteq X$, then $J(X) = X$. We want to show that $J$ is $n$-ary, i.e., that $J = J^{\omega}_{\leq n}$. Let $X$ a subset of $A$. Then it is obvious that $J^{\omega}_{\leq n}(X) = \bigcup_{m\in \mathbb{N}}J^{m}_{\leq n}(X)\subseteq J(X)$. We now proceed to prove that $J(X)\subseteq J^{\omega}_{\leq n}(X)$. Since $J$ is isotone and, by the definition of $J^{\omega}_{\leq n}$, $X\subseteq J^{\omega}_{\leq n}(X)$, we have that $J(X)\subseteq J(J^{\omega}_{\leq n}(X))$. Therefore to prove that $J(X)\subseteq J^{\omega}_{\leq n}(X)$ it suffices to prove that $J(J^{\omega}_{\leq n}(X)) = J^{\omega}_{\leq n}(X)$. But the just stated equation follows from the supposition because, as we will prove next, for every $Z\in \mathrm{Sub}_{\leq n}(J^{\omega}_{\leq n}(X))$, we have that $J(Z)\subseteq J^{\omega}_{\leq n}(X)$. Let $Z$ be a subset of $J^{\omega}_{\leq n}(X)$ such that $\mathrm{card}(Z)\leq n$. Then, for some $\ell\in \mathbb{N}$, $\mathrm{supp}_{S}(Z) = \{s_{0},\ldots,s_{\ell-1}\}$ and, for every $\alpha\in \ell$, there exists an $n_{\alpha}\in \mathbb{N}-1$ such that $Z_{s_{\alpha}} = \{z_{\alpha,0},\ldots,z_{\alpha,n_{\alpha}-1}\}$. Therefore, for every $\alpha\in \ell$ and every $\beta\in n_{\alpha}$ there exists an $m_{\alpha,\beta}\in \mathbb{N}$ such that that $z_{\alpha,\beta}\in J^{m_{\alpha,\beta}}_{\leq n}(X)_{s_{\alpha}}$. Since it may be helpful for the sake of understanding, let us represent the situation just described by the following figure:
$$
\begin{matrix}
z_{0,0}\in J^{m_{0,0}}_{\leq n}(X)_{s_{0}} &\dots & z_{0,n_{0}-1}\in J^{m_{0,n_{0}-1}}_{\leq n}(X)_{s_{0}}\\
\vdots &\ddots & \vdots \\
z_{\ell-1,0}\in J^{m_{\ell-1,0}}_{\leq n}(X)_{s_{\ell-1}} &\dots & z_{\ell-1,n_{\ell-1}-1}\in J^{m_{\ell-1,n_{\ell-1}-1}}_{\leq n}(X)_{s_{\ell-1}}
\end{matrix}
$$
Hence, for every $\alpha\in \ell$ there exists a $\beta_{\alpha}\in n_{\alpha}$ such that $Z_{s_{\alpha}}\subseteq J^{m_{\alpha,\beta_{\alpha}}}_{\leq n}(X)_{s_{\alpha}}$. On the other hand, since the family of many-sorted operators $(J^{m}_{\leq n})_{m\in \mathbb{N}}$ on $A$ is an ascending chain, there exists an $m$ in the set $\{m_{\alpha,\beta_{\alpha}}\mid \alpha\in \ell\}$ such that, for every $\alpha\in \ell$, $J^{m_{\alpha,\beta_{\alpha}}}_{\leq n}\leq J^{m}_{\leq n}$. Thus $Z\subseteq J^{m}_{\leq n}(X)$. Therefore, since, in addition, $\mathrm{card}(Z)\leq n$, we have that $Z\in \mathrm{Sub}_{\leq n}(J^{m}_{\leq n}(X))$. Thus
$$
\textstyle
J(Z)\subseteq J^{m+1}_{\leq n}(X) = J_{\leq n}(J^{m}_{\leq n}(X)) = \bigcup\{J(K)\mid K\in \mathrm{Sub}_{\leq n}(J^{m}_{\leq n}(X))\}.
$$
Consequently $J(Z)\subseteq J^{\omega}_{\leq n}(X)$. Hence $J(X)\subseteq J^{\omega}_{\leq n}(X)$. Whence $J = J^{\omega}_{\leq n}$, which completes the proof.
\end{proof}

We next recall the notion of free monoid on a set and, for a set of sorts $S$, we define, by using the the just mentioned notion, the concept of $S$-sorted signature and, for an $S$-sorted signature $\Sigma$, the concept of $\Sigma$-algebra.

\begin{definition}
Let $S$ be a set of sorts. The \emph{free monoid on} $S$, denoted by $\mathbf{S}^{\star}$, is $(S^{\star},\curlywedge,\lambda)$, where $S^{\star}$, the set of all \emph{words on} $S$, is $\bigcup_{n\in\mathbb{N}}\mathrm{Hom}(n,S)$, the set of all mappings $w\colon n\mor S$ from some $n\in \mathbb{N}$ to $S$, $\curlywedge$, the \emph{concatenation} of words on $S$, is the binary operation on $S^{\star}$ which sends a pair of words $(w,v)$ on $S$ to the mapping $w\curlywedge v$ from $\bb{w}+\bb{v}$ to $S$, where $\bb{w}$ and $\bb{v}$ are the lengths ($\equiv$ domains) of the mappings $w$ and $v$, respectively, defined as follows:
$$
w\bconcat v
\nfunction
{\bb{w}+\bb{v}}{S}
{i}{
\begin{cases}
w_{i}, & \text{if $0\leq i < \bb{w}$;}\\
v_{i-\bb{w}}, & \text{if $\bb{w}\leq i < \bb{w}+\bb{v}$,}
\end{cases}
}
$$
and $\lambda$, the \emph{empty word on} $S$, is the unique mapping $\lambda\colon\varnothing\mor S$.
\end{definition}

\begin{definition}
Let $S$ be a set of sorts. Then an $S$-\emph{sorted signature} is a function $\Sigma$ from $\fmon{S}\bprod S$ to $\boldsymbol{\mathcal{U}}$ which sends a pair $(w,s)\in \fmon{S}\bprod S$ to the set $\Sigma_{w,s}$ of the \emph{formal operations} of \emph{arity} $w$, \emph{sort} (or \emph{coarity}) $s$, and \emph{rank} (or \emph{biarity}) $(w,s)$.
\end{definition}

\begin{definition}
Let $\Sigma$ be an $S$-sorted signature and $A$ an $S$-sorted set. The $\fmon{S}\bprod S$-sorted set of the \emph{finitary
operations on} $A$ is the family $(\mathrm{Hom}(A_{w},A_{s}))_{(w,s)\in\fmon{S}\bprod S}$, where, for every $w\in \fmon{S}$, $A_{w} = \prod_{i\in \bb{w}}A_{w_{i}}$. A \emph{structure of} $\Sigma$-\emph{algebra} \emph{on} $A$ is an $\fmon{S}\bprod S$-mapping $F = (F_{w,s})_{(w,s)\in \fmon{S}\times S}$ from $\Sigma$ to $(\mathrm{Hom}(A_{w},A_{s}))_{(w,s)\in\fmon{S}\bprod S}$. For a pair $(w,s)\in \fmon{S}\times S$ and a formal operation $\sigma\in \Sigma_{w,s}$, in order to simplify the notation, the operation from $A_{w}$ to $A_{s}$ corresponding to $\sigma$ under $F_{w,s}$ will be written as $F_{\sigma}$ instead of $F_{w,s}(\sigma)$. A $\Sigma$-\emph{algebra} is a pair $(A,F)$, abbreviated to $\mathbf{A}$, where $A$ is an $S$-sorted set and $F$ a structure of $\Sigma$-algebra on $A$.
\end{definition}

Since it will be used afterwards, we next define, for a set of sorts $S$ and an $S$-sorted set $A$, the notions of algebraic and of uniform many-sorted closure operator on $A$.

\begin{definition}
A many-sorted closure operator $J$ on an $S$-sorted set $A$ is \emph{algebraic} if, for every $X\subseteq A$, $J(X) = \bigcup_{K\subseteq_{\mathrm{fin}} X}J(K)$, and is \emph{uniform} if, for every $X$, $Y\subseteq A$, if $\mathrm{supp}_{S}(X) = \mathrm{supp}_{S}(Y)$, then $\mathrm{supp}_{S}(J(X)) = \mathrm{supp}_{S}(J(Y))$.
\end{definition}

We next prove that, for a many-sorted closure operator, the property of being $n$-ary is stronger than that of being algebraic.

\begin{proposition}\label{$n$-aryalg}
Let $n$ be a natural number. If a many-sorted closure operator $J$ on an $S$-sorted set $A$ is $n$-ary, then $J$ is an algebraic many-sorted closure operator on $A$.
\end{proposition}

\begin{proof}
Let $J$ be an $n$-ary many-sorted closure operator on an $S$-sorted set $A$ and let $X$ be a subset of $A$. Then, obviously,  $\bigcup_{K\subseteq_{\mathrm{fin}} X}J(K)\subseteq J(X)$. Since $J(X) = J^{\omega}_{\leq n}(X) = \bigcup_{m\in \mathbb{N}}J^{m}_{\leq n}(X)$, to prove that $J(X)\subseteq \bigcup_{K\subseteq_{\mathrm{fin}} X}J(K)$ it suffices to prove that, for every $m\in \mathbb{N}$, $J^{m}_{\leq n}(X)\subseteq \bigcup_{K\subseteq_{\mathrm{fin}} X}J(K)$.

For $m = 0$, since $J^{0}_{\leq n}(X) = X$, we have that $J^{0}_{\leq n}(X)\subseteq \bigcup_{K\subseteq_{\mathrm{fin}} X}J(K)$.

Let $m$ be  $k+1$ with $k\geq 0$ and let us suppose that $J^{k}_{\leq n}(X)\subseteq \bigcup_{K\subseteq_{\mathrm{fin}} X}J(K)$. We want to prove that $J^{k+1}_{\leq n}(X)\subseteq \bigcup_{K\subseteq_{\mathrm{fin}} X}J(K)$. However, by definition,  $J^{k+1}_{\leq n}(X) = \bigcup\{J(Z)\mid Z\in \mathrm{Sub}_{\leq n}(J^{k}_{\leq n}(X))\}$. Thus it suffices to prove that, for every $Z\in \mathrm{Sub}_{\leq n}(J^{k}_{\leq n}(X))$, $J(Z)\subseteq \bigcup_{K\subseteq_{\mathrm{fin}} X}J(K)$. Let $Z$ be a subset of $J^{k}_{\leq n}(X)$ such that $\mathrm{card}(Z)\leq n$. Then, since, by the induction hypothesis, $J^{k}_{\leq n}(X)\subseteq \bigcup_{K\subseteq_{\mathrm{fin}} X}J(K)$, we have that $Z\subseteq \bigcup_{K\subseteq_{\mathrm{fin}} X}J(K)$ and, in addition, that $\mathrm{card}(Z)\leq n$. Hence, for some $\ell\in \mathbb{N}$, $\mathrm{supp}_{S}(Z) = \{s_{0},\ldots,s_{\ell-1}\}$ and, for every $\alpha\in \ell$, there exists an $n_{\alpha}\in \mathbb{N}-1$ such that $Z_{s_{\alpha}} = \{z_{\alpha,0},\ldots,z_{\alpha,n_{\alpha}-1}\}$. Therefore, for every $\alpha\in \ell$ and every $\beta\in n_{\alpha}$ there exists a $K^{\alpha,\beta}\subseteq_{\mathrm{fin}}X$ such that that $z_{\alpha,\beta}\in J(K^{\alpha,\beta})_{s_{\alpha}}$. Since it may be helpful for the sake of understanding, let us represent the situation just described by the following figure:
$$
\begin{matrix}
z_{0,0}\in J(K^{0,0})_{s_{0}} &\dots & z_{0,n_{0}-1}\in J(K^{0,n_{0}-1})_{s_{0}}\\
\vdots &\ddots & \vdots \\
z_{\ell-1,0}\in J(K^{\ell-1,0})_{s_{\ell-1}} &\dots & z_{\ell-1,n_{\ell-1}-1}\in J(K^{\ell-1,n_{\ell-1}-1})_{s_{\ell-1}}
\end{matrix}
$$
Then, for every $\alpha\in \ell$, $Z_{s_{\alpha}}\subseteq J(\bigcup_{\beta\in n_{\alpha}}K^{\alpha,\beta})_{s_{\alpha}}$, where $\bigcup_{\beta\in n_{\alpha}}K^{\alpha,\beta}\subseteq_{\mathrm{fin}}X$. So, for $L = \bigcup_{\alpha\in\ell}\bigcup_{\beta\in n_{\alpha}}K^{\alpha,\beta}$, we have that $L\subseteq_{\mathrm{fin}}X$ and $Z\subseteq J(L)$. Therefore $J(Z)\subseteq J(J(L)) = J(L)\subseteq \bigcup_{K\subseteq_{\mathrm{fin}} X}J(K)$.
\end{proof}

We next define when a subset $X$ of the underlying $S$-sorted set $A$ of a $\Sigma$-algebra $\mathbf{A}$ is closed under an operation $F_{\sigma}$ of $\mathbf{A}$, as well as when $X$ is a subalgebra of $\mathbf{A}$.

\begin{definition}\label{Subalg}
Let $\mathbf{A}$ be a $\Sigma$-algebra and $X\subseteq A$. Let $\sigma$ be a formal operation in $\Sigma_{w,s}$. We say that $X$ is \emph{closed under the operation} $F_{\sigma}\colon A_{w}\mor A_{s}$ if, for every $a\in X_{w}$, $F_{\sigma}(a)\in X_{s}$. We say that $X$ is a \emph{subalgebra} of $\mathbf{A}$ if $X$ is closed under the operations of $\mathbf{A}$. We denote by $\mathrm{Sub}(\mathbf{A})$ the set of all subalgebras of $\mathbf{A}$ (which is an algebraic closure system on $A$).
\end{definition}

\begin{definition}
Let $\mathbf{A}$ be a $\Sigma$-algebra. Then we denote by $\mathrm{Sg}_{\mathbf{A}}$ the many-sorted closure operator on $A$ defined as follows:
$$\textstyle
  \mathrm{Sg}_{\mathbf{A}}\nfunction
  {\mathrm{Sub}(A)}
  {\mathrm{Sub}(A)}
  {X}
  {\bigcap \{\,C\in\mathrm{Sub}(\mathbf{A})\mid X\subseteq C\,\},}.
$$
We call $\mathrm{Sg}_{\mathbf{A}}$ the \emph{subalgebra generating many-sorted operator on} $A$ \emph{determined by} $\mathbf{A}$. For every $X\subseteq A$, we call $\mathrm{Sg}_{\mathbf{A}}(X)$ the \emph{subalgebra of} $\mathbf{A}$ \emph{generated by} $X$. Moreover, if $X\subseteq A$ is such that  $\mathrm{Sg}_{\mathbf{A}}(X) = A$, then we say that $X$ is an $S$-sorted set of \emph{generators} of $\mathbf{A}$, or that $X$ \emph{generates} $\mathbf{A}$. Besides, we say that $\mathbf{A}$ is \emph{finitely generated} if there exists an $S$-sorted  subset $X$ of $A$ such that $X$ \emph{generates} $\mathbf{A}$ and $\mathrm{card}(X)<\aleph_{0}$.
\end{definition}

\begin{proposition}\label{Alg}
Let $\mathbf{A}$ be a $\Sigma$-algebra. Then the many-sorted closure operator $\mathrm{Sg}_{\mathbf{A}}$ on $A$ is algebraic, i.e., for every $S$-sorted subset $X$ of $A$, $\mathrm{Sg}_{\mathbf{A}}(X) =
\bigcup_{K\subseteq_{\mathrm{fin}} X}\mathrm{Sg}_{\mathbf{A}}(K)$.
\end{proposition}


For a $\Sigma$-algebra $\mathbf{A}$ we next provide another, more constructive, description of the algebraic many-sorted closure operator $\Sg_{\mathbf{A}}$, which, in addition, will allow us to state a crucial property of $\Sg_{\mathbf{A}}$. Specifically, that $\Sg_{\mathbf{A}}$ is uniform.

\begin{definition}
Let $\Sigma$ be an $S$-sorted signature and $\mathbf{A}$ a $\Sigma$-algebra.
\begin{enumerate}
\item We denote by $\E_{\mathbf{A}}$ the many-sorted operator on
      $A$ that assigns to an $S$-sorted subset $X$ of
      $A$, $\E_{\mathbf{A}}(X) =
      X\cup\bigl(\;\bigcup_{\sigma\in\Sigma_{\cdot,s}}
      F_{\sigma}[X_{\ard(\sigma)}]\bigr)_{s\in S}$, where, for $s\in
      S$, $\Sigma_{\cdot,s}$ is the set of all many-sorted formal
      operations $\sigma$ such that the coarity of $\sigma$ is $s$ and
      for $\ard(\sigma) = w\in \fmon{S}$, the arity of $\sigma$,
      $X_{\ard(\sigma)} = \prod_{i\in \bb{w}}X_{w_{i}}$.

\item If $X\subseteq A$, then we define the family
      $(\E^{n}_{\mathbf{A}}(X))_{n\in \mathbb{N}}$ in $\mathrm{Sub}(A)$,
      recursively, as follows:
      \begin{align*}
      \E_{\mathbf{A}}^{0}(X) &= X\text{,} \\
      \E_{\mathbf{A}}^{n+1}(X) &=
      \E_{\mathbf{A}}(\E_{\mathbf{A}}^{n}(X)) \text{, $n\geq 0$.}
      \end{align*}
\item We denote by $\E_{\mathbf{A}}^{\omega}$ the many-sorted operator on
      $A$ that assigns to an $S$-sorted subset $X$ of
      $A$, $\E_{\mathbf{A}}^{\omega}(X) = \bigcup_{n\in
      \mathbb{N}}\E_{\mathbf{A}}^{n}(X)$.
\end{enumerate}
\end{definition}

\begin{proposition}\label{Sg}
Let $\mathbf{A}$ be a $\Sigma$-algebra and $X\subseteq A$, then $\Sg_{\mathbf{A}}(X) = \E_{\mathbf{A}}^{\omega}(X)$.
\end{proposition}

In~\cite{cs04}, on pp. 82, we stated the following proposition (there called Proposition 2.7).

\begin{proposition}\label{Unif}
Let $\mathbf{A}$ be a $\Sigma$-algebra and $X, Y\subseteq A$. Then we have that
\begin{enumerate}
\item  If $\mathrm{supp}_{S}(X) = \mathrm{supp}_{S}(Y)$, then, for every $n\in \mathbb{N}$, $\mathrm{supp}_{S}(\mathrm{E}^{n}_{\mathbf{A}}(X)) = \mathrm{supp}_{S}(\mathrm{E}^{n}_{\mathbf{A}}(Y))$.
\item $\mathrm{supp}_{S}(\mathrm{Sg}_{\mathbf{A}}(X)) = \bigcup_{n\in \mathbb{N}}\mathrm{supp}_{S}(\mathrm{E}^{n}_{\mathbf{A}}(X))$.
\item If $\mathrm{supp}_{S}(X) = \mathrm{supp}_{S}(Y)$, then $\mathrm{supp}_{S}(\mathrm{Sg}_{\mathbf{A}}(X)) = \mathrm{supp}_{S}(\mathrm{Sg}_{\mathbf{A}}(Y))$.
\end{enumerate}

Therefore the algebraic many-sorted closure operator $\mathrm{Sg}_{\mathbf{A}}$ is uniform.
\end{proposition}

\begin{proposition}
If $\mathbf{A}$ is a finitely generated $\Sigma$-algebra, then every $S$-sorted set of generators of $\mathbf{A}$ contains a finite $S$-sorted subset which also generates $\mathbf{A}$.
\end{proposition}

\begin{corollary}
If $\mathbf{A}$ is a finitely generated $\Sigma$-algebra, then we have that $\mathrm{IrB}(A,\mathrm{Sg}_{\mathbf{A}})$ is not empty.
\end{corollary}

\section{A characterization of the $n$-ary many-sorted closure operators.}\hfill

A theorem of Birkhoff-Frink (see~\cite{bf48}) asserts that every algebraic closure operator on an ordinary set arises, from some algebraic structure on the set, as the corresponding generated subalgebra operator. However, for many-sorted sets such a theorem is not longer true without qualification. In~\cite{cs04}, on pp. 83--84, Theorem 3.1 and  Corollary 3.2, we characterized the corresponding many-sorted closure operators as precisely the uniform algebraic operators. We next recall the just mentioned characterization since it will be applied afterwards to provide a characterization of the $n$-ary many-sorted closure operators on an $S$-sorted set.

Let us notice that in what follows, for a word $w\colon \lvert w \rvert\rightarrow S$ on $S$, with $\lvert w \rvert$ the lenght of $w$, and an $s\in S$, we denote by $w^{-1}[s]$ the set $\{i\in \lvert w \rvert \mid w(i)=s \}$, and by $\mathrm{Im}(w)$  the set $\{w(i)\mid i\in \lvert w \rvert\}$

\begin{theorem}\label{ThRep}
Let $J$ be an algebraic many-sorted closure operator on an $S$-sorted set $A$. If $J$ is uniform, then $J = \mathrm{Sg}_{\mathbf{A}}$ for some $S$-sorted signature $\Sigma$ and some $\Sigma$-algebra $\mathbf{A}$.
\end{theorem}

\begin{proof}
Let $\Sigma = (\Sigma_{w,s})_{(w,s)\in \fmon{S}\times S}$ be the $S$-sorted signature defined, for every $(w,s)\in\fmon{S}\times S$, as follows:
$$
\Sigma_{w,s} = \{\,(X,b)\in \textstyle{\bigcup}_{X\in \mathrm{Sub}(A)}(\{X\}\times J(X)_{s}) \mid \forall\, t\in S\,(\mathrm{card}(X_{t}) = \lvert w \rvert_{t})\,\},
$$
where for a sort $s\in S$ and a word $w\colon \lvert w \rvert\rightarrow S$ on $S$, with $\lvert w \rvert$ the lenght of $w$, the number of occurrences of $s$ in $w$, denoted by $\lvert w \rvert_{s}$, is $\mathrm{card}(w^{-1}[s])$.

Before proceeding any further, let us remark that, for $(w,s)\in\fmon{S}\times S$ and $(X,b)\in\textstyle{\bigcup}_{X\in \mathrm{Sub}(A)}(\{X\}\times J(X)_{s})$, the following conditions are equivalent:

\begin{enumerate}
\item $(X,b)\in \Sigma_{w,s}$, i.e., for every $t\in S$, $\mathrm{card}(X_{t}) = \lvert w \rvert_{t}$.
\item $\mathrm{supp}_{S}(X) = \mathrm{Im}(w)$ and, for every $t\in\mathrm{supp}_{S}(X)$, $\mathrm{card}(X_{t}) = \lvert w \rvert_{t}$.
\end{enumerate}

On the other hand, for the index set $\Lambda = \bigcup_{Y\in \mathrm{Sub}(A)}(\{Y\}\times \mathrm{supp}_{S}(Y))$ and the $\Lambda$-indexed family $(Y_{s})_{(Y,s)\in\Lambda}$ whose $(Y,s)$-th coordinate is $Y_{s}$, precisely the $s$-th coordinate of the $S$-sorted set $Y$ of the index $(Y,s)\in \Lambda$, let $f$ be a choice function for $(Y_{s})_{(Y,s)\in\Lambda}$, i.e., an element of $\prod_{(Y,s)\in \Lambda}Y_{s}$.

Moreover, for every $w\in\fmon{S}$ and $a\in \prod_{i\in \lvert w \rvert}A_{w(i)}$, let $M^{w,a} = (M^{w,a}_{s})_{s\in S}$ be the finite $S$-sorted subset of $A$ defined as $M^{w,a}_{s} = \{a_{i}\mid i\in w^{-1}[s]\}$, for every $s\in S$.

Now, for $(w,s)\in\fmon{S}\times S$ and $(X,b)\in\Sigma_{w,s}$, let $F_{X,b}$ be the many-sorted operation from $\prod_{i\in\lvert w \rvert}A_{w(i)}$ into $A_{s}$ that to an $a\in \prod_{i\in\lvert w \rvert}A_{w(i)}$ assigns $b$, if $M^{w,a} = X$ and $f(J(M^{w,a}),s)$, otherwise.

We will prove that the $\Sigma$-algebra $\mathbf{A} = (A,F)$ is such that $J = \mathrm{Sg}_{\mathbf{A}}$. But before doing that it is necessary to verify that the definition of the many-sorted operations is sound, i.e., that for every $(w,s)\in\fmon{S}\times S$, $(X,b)\in \Sigma_{w,s}$ and $a\in \prod_{i\in\lvert w \rvert}A_{w(i)}$, it happens that $s\in\mathrm{supp}_{S}(J(M^{w,a}))$,  and for this it suffices to prove that $\mathrm{supp}_{S}(M^{w,a}) = \mathrm{supp}_{S}(X)$, because, by hypothesis, $J$ is uniform and, by definition, $b\in J(X)_{s}$.

If $t\in\mathrm{supp}_{S}(M^{w,a})$, then $M^{w,a}_{t}$ is nonempty, i.e., there exists an $i\in\lvert w \rvert$ such that $w(i) = t$.  Therefore, because $(X,b)\in \Sigma_{w,s}$, we have that $0<\lvert w \rvert_{t} = \mathrm{card}(X_{t})$, hence $t\in\mathrm{supp}_{S}(X)$.

Reciprocally, if $t\in\mathrm{supp}_{S}(X)$, $\lvert w \rvert_{t}>0$, and there is an $i\in\lvert w \rvert$ such that $w(i) = t$, hence $a_{i}\in A_{t}$, and from this we conclude that $M^{w,a}_{t}\neq\varnothing$, i.e., that $t\in\mathrm{supp}_{S}(M^{w,a})$.  Therefore, $\mathrm{supp}_{S}(M^{w,a}) = \mathrm{supp}_{S}(X)$ and, by the uniformity of $J$, $\mathrm{supp}_{S}(J(M^{w,a})) = \mathrm{supp}_{S}(J(X))$. But, by definition, $b\in J(X)_{s}$, so $s\in\mathrm{supp}_{S}(J(M^{w,a}))$ and the definition is sound.

Now we prove that, for every $X\subseteq A$, $J(X)\incl \mathrm{Sg}_{\mathbf{A}}(X)$. Let $X$ be an $S$-sorted subset of $A$, $s\in S$ and $b\in J(X)_{s}$.  Then, because $J$ is algebraic, $b\in J(Y)_{s}$, for some finite $S$-sorted subset $Y$ of $X$. From such an $Y$ we will define a word $w_{Y}$ in $S$ and an element $a_{Y}$ of $\prod_{i\in\lvert w_{Y} \rvert}A_{w_{Y}(i)}$ such that
\begin{enumerate}
\item[(1)] $Y = M^{w_{Y},a_{Y}}$,
\item[(2)] $(Y,b)\in \Sigma_{w_{Y},s}$, i.e., $b\in J(Y)_{s}$ and, for all $t\in S$, $\mathrm{card}(Y_{t}) = \lvert w_{Y} \rvert_{t}$, and
\item[(3)] $a_{Y}\in \prod_{i\in\lvert w_{Y} \rvert}X_{w_{Y}(i)}$,
\end{enumerate}
then, because $F_{Y,b}(a_{Y}) = b$, we will be entitled to assert that $b\in \mathrm{Sg}_{\mathbf{A}}(X)_{s}$.

But given that $Y$ is finite if, and only if, $\mathrm{supp}_{S}(Y)$ is finite and, for every $t\in \mathrm{supp}_{S}(Y)$, $Y_{t}$ is finite, let $\{\,s_{\alpha}\mid \alpha\in m\,\}$ be an enumeration of $\mathrm{supp}_{S}(Y)$ and, for every $\alpha\in m$, let $\{\,y_{\alpha,i}\mid i\in p_{\alpha}\,\}$ be an enumeration of the nonempty $s_{\alpha}$-th coordinate, $Y_{s_{\alpha}}$, of $Y$.  Then we define, on the one hand, the word $w_{Y}$ as the mapping from $\lvert w_{Y} \rvert = \sum_{\alpha\in m}p_{\alpha}$ into $S$ such that, for every $i\in \lvert w_{Y} \rvert$ and $\alpha\in m$, $w_{Y}(i) = s_{\alpha}$ if, and only if, $\sum_{\beta\in \alpha}p_{\beta}\leq i\leq \sum_{\beta\in \alpha+1}p_{\beta}-1$ and, on the other hand, the element $a_{Y}$ of $\prod_{i\in\lvert w_{Y} \rvert}A_{w_{Y}(i)}$ as the mapping from $\lvert w_{Y} \rvert$ into $\bigcup_{i\in \lvert w_{Y} \rvert}A_{w_{Y}(i)}$ such that, for every $i\in \lvert w_{Y} \rvert$ and $\alpha\in m$, $a_{Y}(i) = y_{\alpha,i-\sum_{\beta\in \alpha}p_{\beta}}$ if, and only if, $\sum_{\beta\in \alpha}p_{\beta}\leq i\leq \sum_{\beta\in \alpha+1}p_{\beta}-1$.  From these definitions follow (1), (2) and (3) above. Let us observe that (1) is a particular case of the fact that the mapping $M$ from $\bigcup_{w\in\fmon{S}}(\{w\}\times \prod_{i\in\lvert w \rvert}A_{w(i)})$ into $\mathrm{Sub}_{\text{fin}}(A)$ that to a pair $(w,a)$ assigns $M^{w,a}$ is surjective.

From the above and the definition of $F_{Y,b}$ we can affirm that $F_{Y,b}(a_{Y}) = b$, hence $b\in \mathrm{Sg}_{\mathbf{A}}(X)_{s}$.  Therefore $J(X)\incl\mathrm{Sg}_{\mathbf{A}}(X)$.

Finally, we prove that, for every $X\subseteq A$, $\mathrm{Sg}_{\mathbf{A}}(X)\incl J(X)$.  But for this, by Proposition~\ref{Sg}, it is enough to prove that, for every subset $X$ of $A$, we have that $\mathrm{E}_{\mathbf{A}}(X)\incl J(X)$.  Let $s\in S$ be and $c\in\mathrm{E}_{\mathbf{A}}(X)_{s}$. If $c\in X_{s}$, then $c\in J(X)_{s}$, because $J$ is extensive. If $c\nin X_{s}$, then, by the definition of $\mathrm{E}_{\mathbf{A}}(X)$, there exists a word $w\in\fmon{S}$, a many-sorted formal operation $(Y,b)\in\Sigma_{w,s}$ and an $a\in \prod_{i\in \lvert w \rvert}X_{w(i)}$ such that $F_{Y,b}(a) = c$. If $M^{w,a} = Y$, then $c = b$, hence $c\in J(Y)_{s}$, therefore, because $M^{w,a}\subseteq X$, $c\in J(X)_{s}$.  If $M^{w,a}\neq Y$, then $F_{Y,b}(a)\in J(M^{w,a})_{s}$, but, because $M^{w,a}\incl X$ and $J$ is isotone, $J(M^{w,a})$ is a subset of $J(X)$, hence $F_{Y,b}(a)\in J(X)_{s}$. Therefore $\mathrm{E}_{\mathbf{A}}(X)\incl J(X)$.
\end{proof}

The just stated theorem together with Proposition~\ref{Unif} entails the following corollary.

\begin{corollary}
Let $J$ be an algebraic many-sorted closure operator on an $S$-sorted set $A$. Then $J = \mathrm{Sg}_{\mathbf{A}}$ for some $S$-sorted signature $\Sigma$ and some $\Sigma$-algebra $\mathbf{A}$ if, and only if, $J$ is uniform.
\end{corollary}

We next prove that for a natural number $n$, an $S$-sorted signature $\Sigma$, and a $\Sigma$-algebra $\mathbf{A}$, under a suitable condition on $\Sigma$ related to $n$, the uniform algebraic many-sorted closure operator $\mathrm{Sg}_{\mathbf{A}}$ is an $n$-ary many-sorted closure operator on $A$.

\begin{proposition}\label{Arity}
Let $\Sigma$ be an $S$-sorted signature, $\mathbf{A}$ a $\Sigma$-algebra, and $n\in \mathbb{N}$. If $\Sigma$ is such that, for every $(w,s)\in S^{\star}\times S$, $\Sigma_{w,s} = \varnothing$ if $\bb{w}> n$---in which case we will say that every operation of $\mathbf{A}$ is of an arity $\leq n$---, then the uniform algebraic many-sorted closure operator $\mathrm{Sg}_{\mathbf{A}}$ is an $n$-ary many-sorted closure operator on $A$, i.e., $\mathrm{Sg}_{\mathbf{A}} = (\mathrm{Sg}_{\mathbf{A}})_{\leq n}^{\omega}$.
\end{proposition}

\begin{proof}
It follows from $\Sg_{\mathbf{A}}(X) = \E_{\mathbf{A}}^{\omega}(X)$ and from the fact that, for every $X\subseteq A$,  $\E_{\mathbf{A}}(X)\subseteq (\mathrm{Sg}_{\mathbf{A}})_{\leq n}(X)\subseteq \mathrm{Sg}_{\mathbf{A}}(X)$. The details are left to the reader. However, we notice that it is advisable to split the proof into two cases, one for $n = 0$ and another one for $n\geq 1$.
\end{proof}

\begin{proposition}
Let $A$ be an $S$-sorted set, $J$ a many-sorted closure operator on $A$, and $n\in \mathbb{N}$. If $J$ is $n$-ary (hence, by Proposition \ref{$n$-aryalg}, algebraic) and uniform, then there exists an $S$-sorted signature $\Sigma'$ and a $\Sigma'$-algebra $\mathbf{A}'$ such that $J = \mathrm{Sg}_{\mathbf{A}'}$ and every operation of $\mathbf{A}'$ is of an arity $\leq n$.
\end{proposition}

\begin{proof}
If we denote by $\mathbf{A} = (A,F)$ the $\Sigma$-algebra associated to $J$ constructed in the proof of Theorem~\ref{ThRep}, then taking as $\Sigma'$ the $S$-sorted signature defined, for every $(w,s)\in S^{\star}\times S$, as: $\Sigma'_{w,s} = \Sigma_{w,s}$, if $\bb{w}\leq n$; and $\Sigma'_{w,s} = \varnothing$, if $\bb{w} > n$, and as $\mathbf{A}' = (A',F')$ the $\Sigma'$-algebra defined as: $A' = A$, and $F' = F\circ \mathrm{inc}^{\Sigma',\Sigma}$, where $\mathrm{inc}^{\Sigma',\Sigma} = (\mathrm{inc}^{\Sigma',\Sigma}_{w,s})_{(w,s)\in S^{\star}\times S}$ is the canonical inclusion of $\Sigma'$ into $\Sigma$, then one can show that $J = \mathrm{Sg}_{\mathbf{A}'}$.
\end{proof}

From the just stated proposition together with Proposition~\ref{Arity} it follows immediately the following corollary, which is an  algebraic characterization of the $n$-ary and uniform many-sorted closure operators.

\begin{corollary}
Let $J$ be a many-sorted closure operator on an $S$-sorted set $A$ and $n\in \mathbb{N}$. Then $J$ is $n$-ary and uniform if, and only if, there exists an  $S$-sorted signature $\Sigma$ and a $\Sigma$-algebra $\mathbf{A}$ such that $J = \mathrm{Sg}_{\mathbf{A}}$ and every operation of $\mathbf{A}$ is of an arity $\leq n$.
\end{corollary}

\section{The irredundant basis theorem for many-sorted closure spaces.}

We next show Tarski's irredundant basis theorem for many-sorted closure spaces.

\begin{theorem}[Tarski's irredundant basis theorem for many-sorted closure spaces]
Let $(A,J)$ be a  many-sorted closure space. If $J$ is an $n$-ary many-sorted operator on the $S$-sorted set $A$, with $n\geq 2$, and if $i<j$ with $i,j\in \mathrm{IrB}_{J}(A)$ such that
$$
 \{i+1,\ldots,j-1\}\cap \mathrm{IrB}_{J}(A) = \vacio,
$$
then $j-i\leq n-1$. In particular, if $n = 2$, then $\mathrm{IrB}_{J}(A)$ is a convex subset of $\mathbb{N}$.
\end{theorem}

\begin{proof}
Let $Z\subseteq A$ be an irredundant basis with respect to $J$ such that $\mathrm{card}(Z) = j$ and $\mathcal{K} = \{\,X\in
\mathrm{IrB}_{J}(A)\mid \mathrm{card}(X)\leq i\,\}$. Since $J$ is $n$-ary, we can assert that $J(Z) = A = \bigcup_{m\in
\mathbb{N}}J^{m}_{\leq n}(Z)$, so, for every $s\in S$, $J(Z)_{s} = A_{s} = \bigcup_{m\in \mathbb{N}}J^{m}_{\leq n}(Z)_{s}$. Let $X$ be an element of $\mathcal{K}$. Then there exists a $k\in \mathbb{N}-1$ such that $X\subseteq J^{k}_{\leq n}(Z)$. The natural number $k$ should be strictly greater than $0$, because if $k = 0$, $X\subseteq J^{0}_{\leq n}(Z) = Z$, but $\mathrm{card}(X) = i<j = \mathrm{card}(Z)$, so $Z$ would not be an irredundant basis. So that, for every $X\in
\mathcal{K}$, $\{\,k\in \mathbb{N}-1\mid X\subseteq J^{k}_{\leq n}(Z)\,\}\neq \vacio$. Therefore, for every $X\in \mathcal{K}$,
we can choose the least element of such a set, denoted by $d_{Z}(X)$, and there is fulfilled that $d_{Z}(X)$ is greater than or equal to $1$. For $d_{Z}(X)-1$ we have that $X\nsubseteq J^{d_{Z}(X)-1}_{\leq n}(Z)$. So we conclude that there exists a mapping  $d_{Z}\colon\mathcal{K}\mor \mathbb{N}-1$ that to an $X\in \mathcal{K}$ assigns $d_{Z}(X)$. The image of the mapping $d_{Z}$, which is a nonempty part of $\mathbb{N}-1$, is well-ordered, hence it has a least element, which is, necessarily, non zero, $t+1$, therefore, since $\mathcal{K}/\mathrm{Ker}(d_{Z})$ is isomorphic to $\mathrm{Im}(d_{Z})$, by transport of structure, it will also be well-ordered, then we can always choose an $X\in \mathcal{K}$ such that, for every $Y\in \mathcal{K}$, $d_{Z}(X)\leq d_{Z}(Y)$,
e.g., an $X$ such that its equivalence class corresponds to the minimum $t+1$ of $\mathrm{Im}(d_{Z})$.  Moreover, among the $X$
which have the just mentioned property, we choose an $X^{0}$ such that, for every $Y\in \mathcal{K}$ with $Y\subseteq J^{t+1}_{\leq n}(Z)$, it happens that
$$
 \mathrm{card}(X^{0}\cap(J^{t+1}_{\leq n}(Z)-J^{t}_{\leq n}(Z)))\leq
 \mathrm{card}(Y\cap(J^{t+1}_{\leq n}(Z)-J^{t}_{\leq n}(Z))).
$$

By the method of election we have that $X^{0}\subseteq J^{t+1}_{\leq n}(Z)$ but $X^{0}\nsubseteq J^{t}_{\leq n}(Z)$. Of the latter we conclude that there exists an $s_{0}\in S$ such that $X^{0}_{s_{0}}\nsubseteq J^{t}_{\leq n}(Z)_{s_{0}}$, therefore
$$
(J^{t+1}_{\leq n}(Z)_{s_{0}}-J^{t}_{\leq n}(Z)_{s_{0}})\cap
X^{0}_{s_{0}}\neq \vacio.
$$
Let $a_{0}\in (J^{t+1}_{\leq n}(Z)_{s_{0}}-J^{t}_{\leq n}(Z)_{s_{0}})\cap X^{0}_{s_{0}}$ be. Then $a_{0}\in
X^{0}_{s_{0}}$, $a_{0}\in J^{t+1}_{\leq n}(Z)_{s_{0}}$ but $a_{0}\nin J^{t}_{\leq n}(Z)_{s_{0}}$. However, $J^{t+1}_{\leq
n}(Z) = J_{\leq n}(J^{t}_{\leq n}(Z))$, by definition, hence there exists a part $F$ of $J^{t}_{\leq n}(Z)$ such that $\mathrm{card}(F)\leq n$ and $a_{0}\in J(F)_{s_{0}}$. Let $X^{1}$ be the part of $A$ defined as follows:
\begin{equation*}
X^{1}_{s} =
  \begin{cases}
   X^{0}_{s}\cup F_{s}, & \text{if $s\neq s_{0}$;}\\
   (X^{0}_{s_{0}}-\{a_{0}\})\cup F_{s_{0}},
   & \text{if $s = s_{0}$.}
\end{cases}
\end{equation*}

It holds that $X^{0}\subseteq J(X^{1})$. Therefore $J(X^{0})\subseteq J(X^{1})$, but $J(X^{0}) = A$, hence $J(X^{1})
= A$, i.e., $X^{1}$ is a finite generator with respect to $J$, thus $X^{1}$ will contain a minimal generator $X^{2}$ with respect to $J$. It holds that $\mathrm{card}(X^{2})\leq \mathrm{card}(X^{1})<\mathrm{card}(X^{0})+n$. It cannot happen that
$\mathrm{card}(X^{0})+n\leq j$. Because if $\mathrm{card}(X^{0})+n\leq j$, then $\mathrm{card}(X^{2})<j$,
hence, since
$$
 \{i+1,\ldots,j-1\}\cap \mathrm{IrB}(A,J) = \vacio,
$$
$X^{2}\in \mathcal{K}$, but $X^{2}\subseteq J^{t+1}_{\leq n}(Z)$ and, moreover, it happens that
$$
\mathrm{card}(X^{2}\cap(J^{t+1}_{\leq n}(Z)-J^{t}_{\leq n}(Z))) <
\mathrm{card}(X^{0}\cap(J^{t+1}_{\leq n}(Z)-J^{t}_{\leq n}(Z))),
$$
because $a_{0}\nin X^{2}_{s_{0}}$ but $a_{0}\in X^{0}_{s_{0}}$, which contradicts the choice of $X^{0}$. Hence
$\mathrm{card}(X^{0})+n>j$. But $\mathrm{card}(X^{0})\leq i$, therefore $j-i<n$, i.e., $j-i\leq n-1$.
\end{proof}


\end{document}